\documentclass[12pt]{article}
\usepackage{amssymb,amsmath,amsthm,amscd,tikz}

\title{Extension of Convex Function}  
\author{Min Yan\\ Department of Mathematics \\ Hong Kong University of Science and Technology} 

\newcommand{\sub}{\subset}
\newcommand{\pa}{\partial}

\newcommand{\bb}{\mathbb}

\newtheorem{theorem}{Theorem}[section]

\newtheorem{proposition}[theorem]{Proposition}

\theoremstyle{definition}

\newtheorem*{definition*}{Definition}
\newtheorem*{conjecture*}{Conjecture}
\newtheorem{example}{Example}[section]

\begin{document}
\maketitle

\begin{abstract}
We study the local and global versions of the convexity, which is closely related to the problem of extending a convex function on a non-convex domain to a convex function on the convex hull of the domain and beyond the convex hull. We also give the parallel results for the convexity defined by positive definite Hessian.  
\end{abstract}

%\tableofcontents



%\newpage

%\renewcommand{\baselinestretch}{1.5} \large\normalsize

\section{Introduction}

Convex functions appear in many important problems in pure and applied mathematics. In many literatures on convex analysis \cite{bv,Phelps,Rockafellar}, convex functions are usually defined only on convex domains. Moreover, convex functions are often extended to the whole linear space by setting the value to be $+\infty$ out of the convex domain, so that the extended function is still convex. While such treatment is preferred in some applications (e.g., optimization, convex programming) and some theories (e.g., duality), it may not be the desirable thing to do if the problem is more analytic (e.g., variational problems). 

In more analytic applications, the convexity of a $C^2$-function $f$ often means that its Hessian $H_f(v)=\sum\frac{\pa^2 f}{\pa x_i\pa x_j}v_iv_j$ is a positive semidefinite (or positive definite for strict convexity) quadratic form of vectors $v=(v_1,\dots,v_n)$. This ``Hessian approach'' to convexity is clearly local and does not require the domain to be convex. Moreover, the differentiable function cannot take the infinity value. Therefore on a non-convex domain, there are distinct global and local versions of the convexity, and it is worthwhile to study the relation between the two. The natural connection between the two is the question whether a convex function on a non-convex domain can be extended to a convex function on the convex hull of the domain. Moreover, in many applications, it is also desirable to know whether a convex function can be extended beyond the convex hull.

The extension of convex functions has wide ranging applications in geometric analysis \cite{tak}, nonlinear dynamics \cite{vv}, quantum computing \cite{uh1,uh2,wo} and economics \cite{pw}. Peters and Wakker \cite{pw}, in their study of decision making under risk, gave the necessary and sufficient condition for extending a function on a non-convex domain to a convex function on the convex hull of the domain. Dragomirescu and Ivan \cite{di} constructed the minimal convex extension of a convex function on a convex domain to the whole linear space. Researchers also studied the problem of extending a function on a subset of the boundary of a convex subset to a convex function on the convex subset \cite{bl,lima,rz}. All these works did not address the extension of the Hessian version of convexity.

The purpose of this paper is to clarify the global and local versions of the convexity and to study the extensions of convex functions. We emphasize that the functions and their extensions do not take the infinity value in this paper. We also study the extension of the Hessian version of convexity. We restrict ourselves to the Euclidean space, although some results remain valid in more general locally convex topological vector spaces.

In Section \ref{convex}, we introduce three types of convexities (and their strict versions): Convex, locally convex, and interval convex. As explained above, the first two concepts address the core issue concerning us. The third concept is included because it is the weakest version of convexity and plays a very useful technical role. We discuss the relation between various convexities and show that they are generally not equivalent. 

In Section \ref{hull}, we reformulate the results by Peters and Wakker into Theorem \ref{tohull} that gives convex extensions avoiding the infinity value. In Theorem \ref{tohulldiff}, we show the similar extension problem for $C^2$-functions with positive definite Hessian can be solved as long as the general convex extension problem can be solved.

In Section \ref{outside}, we try to extend the convexity out of the convex hull of the domain, while avoiding the infinity value. Theorem \ref{toall} says that, the condition for extending the convexity of a function on a bounded convex domain to the whole linear space is exactly the Lipschitz property. Theorem \ref{toalldiff} gives one case that Theorem \ref{tohull} (the extension to the convex hull) and Theorem \ref{toall} (the extension out of the convex hull) can be combined. Theorems \ref{hessian} and \ref{hessian2} give the Hessian version of the extension out of the convex hull. Since the Hessian convexity is local, we do not require the domain $\Omega$ to be convex in general, and the extension is actually to a subset slightly smaller than $\Omega\cup({\bb R}^n-\Omega^{\text{co}})$.

In Section \ref{compare}, we study in more detail the equivalence between various convexities. We conclude that, roughly speaking, if the domain is convex up to deleting a subset of codimension $\ge 2$, then all convexities are equivalent. However, the equivalence no longer holds if the domain is only convex up to deleting a subset of codimension $1$. 

Throughout the paper, all functions take ordinary values and never take the infinity value. Moreover, $\Omega^{\text{co}}$, $\Omega^{\text{aff}}$, $\Omega^{\text{ri}}$, $\Omega^{\epsilon}$ denote the convex hull, the affine hull, the relative interior, and the $\epsilon$-neighborhood of a subset $\Omega$. 

Finally, for more practical purpose, such as certain specific estimations in fluid dynamics, it would be even more useful if we can control the size of the Hessian of the extension. This leads to the following conjecture.

\begin{conjecture*}
Suppose $f$ is a $C^2$-function on a compact convex subset $\Omega$, such that the Hessian satisfies $A\|v\|^2<H_f(v)< B\|v\|^2$ for some constants $A$ and $B$. Then $f$ can be extended to a $C^2$-function on the whole linear space such that the Hessian satisfies the same bound.
\end{conjecture*}

Theorem \ref{hessian2} gives the affirmative answer to the special case $A=0$ and $B=+\infty$.

The author would like to thank the referee for many suggestions, especially simpler versions of examples.

\section{Convex, locally convex, and interval convex}
\label{convex}

A {\em convex combination} of several points $x_1,\dotsc,x_k\in {\bb R}^n$ is
\begin{equation}\label{ccombo}
x=\lambda_1 x_1+\cdots+\lambda_k x_k,\quad
\lambda_i\geq 0,\; \lambda_1+\cdots+\lambda_k=1.
\end{equation}
The {\em convex hull} $\Omega^{\text{co}}$ of a subset $\Omega\sub {\bb R}^n$ is the collection of all convex combinations of all points in $\Omega$. The subset is {\em convex} if and only if $\Omega^{\text{co}}=\Omega$.

A function $f$ on $\Omega$ is usually defined as convex if 
\begin{equation}\label{convexf}
f(x)\le \lambda_1 f(x_1)+\cdots+\lambda_k f(x_k)
\end{equation}
for any convex combination \eqref{ccombo}. However, if $\Omega$ is not convex, then we may have $x_i\in\Omega$ but $x\not\in\Omega$. We clarify the definition for not necessarily convex domain by introducing different kinds of convexities.

\begin{definition*}
Let $f$ be a function with domain $\Omega\sub{\bb R}^n$.
\begin{enumerate}
\item $f$ is {\em convex} if \eqref{convexf} is satisfied whenever $x_1,\dotsc,x_k\in\Omega$ and their convex combination $x=\lambda_1 x_1+\cdots+\lambda_k x_k\in\Omega$. 
\item $f$ is {\em locally convex} if at any $x\in\Omega$, there is a ball $B$ around $x$, such that the restriction of $f$ to $B\cap\Omega$ is convex.
\item $f$ is {\em interval convex} if the restriction of $f$ to any interval (i.e., straight line segment) inside $\Omega$ is convex.
\end{enumerate}
If the inequality in \eqref{convexf} is changed to strict inequality when all $x_i\ne x$, then we get the {\em strict} version of various convexities.
\end{definition*}

The convexity is the global property that we try to achieve. 

The local convexity is introduced to accommodate the second derivative test, or the Hessian version of the convexity. So a typical example is a $C^2$-function on an open domain, such that the Hessian is positive semidefinite everywhere. The classical Busemann-Feller-Aleksandrov theorem \cite{al,bf} says that the local convexity on an open domain is not far off from the Hessian convexity.

The interval convexity means
\begin{equation}\label{convexi}
f(\lambda_1x_1+\lambda_2x_2)\le \lambda_1 f(x_1)+\lambda_2 f(x_2) 
\text{ for any }\lambda_1x_1+\lambda_2x_2\in [x_1,x_2]\sub \Omega.
\end{equation}
Here we use $[x_1,x_2]$ to denote the interval between $x_1$ and $x_2$. We introduce the interval convexity because it is the weakest form of convexity. Therefore if the interval convexity implies the convexity, then many other kinds of convexities will also imply the convexity.

The three kinds of convexities are related by
\[
\text{convex}
\implies \text{locally convex}
\implies \text{interval convex}.
\]
The first implication is due to the fact that the convexity on $\Omega$ implies the convexity on subsets of $\Omega$. For the second implication, we first note that it is sufficient to consider compact intervals only. Suppose $f$ is locally convex and $I\sub \Omega$ is a compact interval. Then we have $I=I_1\cup \dots\cup I_p$, such that $I_i\cap I_{i+1}$ are open intervals in $I$, and $f$ is convex on each $I_i$. We know that, if a single variable function is convex on each of two intervals that overlap at more than one point, then the function is convex on the union of the two intervals. This enables us to conclude that $f$ is convex on $I$. 

It is well known that, on a convex domain $\Omega$, the usual convexity is equivalent to the property \eqref{convexf} for the special case $k=2$ (the combination $\lambda_1x_1+\lambda_2x_2$ is convex)
\begin{equation}\label{convexf2}
f(\lambda_1x_1+\lambda_2x_2)\le \lambda_1 f(x_1)+\lambda_2 f(x_2) 
\text{ for any }x_1,x_2,\lambda_1x_1+\lambda_2x_2\in \Omega.
\end{equation}
Since the convexity of $\Omega$ implies that the interval $[x_1,x_2]$ is automatically contained in $\Omega$, the property \eqref{convexf2} is the same as \eqref{convexi}. This shows that for convex domains, the interval convexity is the same as the convexity, so that the three kinds of convexities are the same. 

If $\Omega$ is {\em locally convex} in the sense that for any $x\in\Omega$, there is a ball $B$ around $x$, such that $B\cap \Omega$ is convex, then the local convexity and the interval convexity are the same. Open subsets are examples of such $\Omega$.

For general not necessarily convex $\Omega$, the property \eqref{convexf2} actually means that $f$ is convex on $L\cap\Omega$ for any line $L$. It is tempting to define the property as the 1-dimensional convexity (and further explore the $m$-dimensional convexity), but we already have enough number of convexities. The property \eqref{convexf2} is generally stronger than the interval convexity  \eqref{convexi}, and actually implies the convexity when the domain is open.

\begin{proposition}\label{lconvex}
A function on an open subset $\Omega$ is convex if and only if the function is convex on $L\cap\Omega$ for any straight line $L$. The strict version is also true.
\end{proposition}

\begin{proof}
The necessity follows from the fact the (strict) convexity on $\Omega$ implies the (strict) convexity on any subset of $\Omega$. It remains to prove the sufficiency.

Consider a convex combination \eqref{ccombo} with all $x_i\ne x$. Since $\Omega$ is open, we have a small ball $B\sub \Omega$ around $x$. Then $f$ is convex on $L\cap B$ for any straight line $L$. As explained above, since $B$ is a convex subset, we conclude that $f$ is convex on $B$. 

Let $L_i$ be the line connecting $x$ and $x_i$. See Fig.~\ref{lconvex_fig}. For a small $\delta>0$, we have $y_i=\delta x_i+(1-\delta)x\in L_i\cap B$. Since $f$ is convex on $L_i\cap \Omega$, we have
\begin{equation}\label{lconvex_eq1}
f(y_i)\le \delta f(x_i)+(1-\delta)f(x).
\end{equation}
Since $f$ is convex on $B$, $y_i\in B$, and $x=\lambda_1y_1+\dotsb+\lambda_ky_k\in B$, we also have
\begin{align}
f(x) 
& \le \lambda_1 f(y_1)+\cdots+\lambda_k f(y_k)  \nonumber \\
& \le \lambda_1 [\delta f(x_1)+(1-\delta)f(x)]+\dotsb+\lambda_k [\delta f(x_k)+(1-\delta)f(x)] \nonumber \\
& = \delta [\lambda_1 f(x_1)+\cdots+\lambda_k f(x_k)]+(1-\delta)f(x). \label{lconvex_eq2}
\end{align} 
This is the same as \eqref{convexf}.

For the strict version, the inequality \eqref{lconvex_eq1} becomes strict, so that \eqref{lconvex_eq2} is overall a strict inequality.
\end{proof}

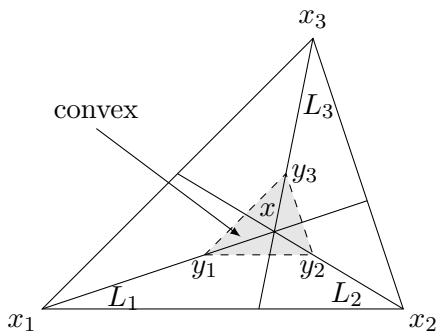
\begin{figure}[h]
\centering
\begin{tikzpicture}[scale=1.2,>=latex]

	\draw
		(0,0) node[below left=-2] {\small $x_1$} -- 
		(4,0) node[below right=-2] {\small $x_2$} -- 
		(3,3) node[above] {\small $x_3$} -- cycle;
	\filldraw[dashed, fill=gray!20, shift={(1.8cm,0.6cm)},scale=0.3]
		(0,0) node[below=-2] {\small $y_1$} -- 
		(4,0) node[below=-2] {\small $y_2$} -- 
		(3,3) node[right=-2] {\small $y_3$} -- cycle;
	\draw
		(0,0) -- node[below=-2.5, near start] {\small $L_1$} (3.6,1.2)
		(4,0) -- node[below=-1.5, near start] {\small $L_2$} (1.5,1.5)
		(3,3) -- node[right=-3, near start] {\small $L_3$} (2.4,0);
	\node at (2.5,1.1) {\small $x$};
	\draw[->]
		(0.6,2) node[above] {\small convex} -- (2.2,0.8);

\end{tikzpicture}
\caption{$1$-dimensional convexity on open subsets implies convexity.}
\label{lconvex_fig}
\end{figure}

\begin{example}\label{eg20}
Suppose $\Omega$ consists of three vertices of a triangle and an interior point of the triangle. Assigning any four values gives a function that is convex on any $L\cap\Omega$. For the function to be convex, however, the four values must satisfy one non-trivial relation. Therefore the open condition in Proposition \ref{lconvex} is necessary. 
\end{example}

\begin{example}\label{eg21}
The local convexity does not generally imply the convexity. For example, given convex functions on disjoint open convex domains $\Omega_1$ and $\Omega_2$, the combined function on $\Omega=\Omega_1\cup\Omega_2$ is still locally convex, but not necessarily convex. A specific counterexample is
\[
f(x)=\begin{cases}
1, & \text{if }x>0, \\
-1, & \text{if }x<0.
\end{cases}
\]

The example has non-connected domain. For a counterexample on a connected domain, see Example \ref{eg22}.
\end{example}

\begin{example}\label{eg23}
The interval convexity does not generally imply the local convexity. Since we do have the implication for open subsets, counterexamples can be constructed only for domains of lower dimension. 

Consider the subset $\Omega=\{re^{i\theta}\colon r\ge 0,\;\theta=0,\frac{2}{3}\pi,\frac{4}{3}\pi\}$ of the plane. The function $(r-1)^2$ expressed in the polar coordinates is convex on each of the three branches of $\Omega$ and is therefore interval convex on $\Omega$. However, the function is not locally convex at $0$ because the convexity fails for the convex combination $0=\frac{1}{3}r+\frac{1}{3}re^{i\frac{2}{3}\pi}+\frac{1}{3}re^{i\frac{4}{3}\pi}$, $0<r<2$. 
\end{example}

By similar thinking, we have the implications
\[
\text{strictly convex}
\implies \text{strictly locally convex}
\implies \text{strictly interval convex}.
\]
On a convex subset, the three kinds of strict convexities are equivalent. Moreover, on a locally convex subset, the strict local convexity and the strict interval convexity are the same. The idea of Example \ref{eg21} (and Example \ref{eg22}) shows that the strict local convexity does not generally imply the strict convexity.  Examples \ref{eg23} shows that the strict interval convexity does not generally imply the strict local convexity. 

The following result shows that, on open domains, the ``strict'' part of the strict convexity is a local requirement.

\begin{proposition}\label{sconvex}
A function on an open subset is strictly convex if and only if it is convex and strictly locally convex.
\end{proposition}

\begin{proof}
The proof follows the same idea as Proposition \ref{lconvex}. By the strict local convexity, $f$ is strictly convex on a ball $B\sub \Omega$ around $x$. We have \eqref{lconvex_eq1} by the convexity of $f$ on $\Omega$. We also have \eqref{lconvex_eq2}, in which the first inequality is strict by the strict convexity of $f$ on $B$. Then we get \eqref{convexf} with strict inequality.
\end{proof}

\begin{example}\label{eg24}
The open condition in Proposition \ref{sconvex} is necessary. Consider the subset $\Omega$ of the plane consisting of the intervals $[-1,1]\times \{1\}$, $[-1,1]\times \{-1\}$ and the origin $(0,0)$. The function $x^2$ is convex on the whole plane and is therefore convex on $\Omega$. The function is also locally strictly convex on $\Omega$ because it is strictly convex on the two intervals. However, the functions fails the strict convexity for the convex combination $(0,0)=\frac{1}{2}(0,1)+\frac{1}{2}(0,-1)$. 
\end{example}

\section{Extension to the Convex Hull}
\label{hull}

The following is the well known result \cite[Theorem 1 and Corollary 2]{pw} about extending the convexity to the convex hull. We modify and extend the result to avoid the infinity value.

\begin{theorem}\label{tohull}
Suppose $f$ is a convex function on $\Omega$. Then under either of the following assumptions, $f$ can be extended to a convex function on the convex hull $\Omega^{\text{\rm co}}$.
\begin{enumerate}
\item $f$ is bounded below.
\item $\Omega$ contains a point in the relative interior of the convex hull $\Omega^{\text{\rm co}}$.
\end{enumerate}
\end{theorem}

\begin{proof}
If $f$ is extended to the convex hull, and $x\in \Omega^{\text{co}}$ is expressed as a convex combination \eqref{ccombo} with $x_i\in\Omega$, then \eqref{convexf} gives an upper bound for the value of the extension at $x$. So it is natural to take the infimum of all such upper bounds. Therefore, to adopt a terminology from quantum computing, we introduce the {\em convex roof}
\[
\hat{f}(x)
=\inf\{\lambda_1 f(x_1)+\dotsb+\lambda_k f(x_k)\},\quad x\in \Omega^{\text{co}},
\]
where the infimum runs over all the possible convex combinations \eqref{ccombo} with $x_i\in \Omega$. The convex roof can be constructed for any function (but with possible $-\infty$ value), and is actually the biggest convex function on $\Omega^{\text{co}}$ satisfying $\hat{f}\le f$ on $\Omega$. It is easy to see that $\hat{f}$ extends $f$ when $f$ is already convex.

The first assumption implies that $\hat{f}$ does not take the infinity value. 

Under the second assumption, let $\Omega^{\text{aff}}$ be the affine span of $\Omega$ and $d=\dim\Omega^{\text{co}}=\dim \Omega^{\text{aff}}$. The assumption implies that there are (necessarily affinely independent) $x_0,\dotsc,x_d\in \Omega$ that affinely span $\Omega^{\text{aff}}$, and there is $\bar{x}\in\Omega$ lying in the relative interior of the convex hull of $x_0,\dotsc,x_d$. 

For any $x\in\Omega^{\text{co}}-\{\bar{x},x_0,\dotsc,x_d\}$, we can always find $0\le i_1<\dotsb<i_p\le d$, such that $x,x_{i_1},\dotsc,x_{i_p}$ are affinely independent and $\bar{x}$ is in the relative interior of the convex hull of $x,x_{i_1},\dotsc,x_{i_p}$. The collection $\{x,x_{i_1},\dotsc,x_{i_p}\}$ depends on the location of $x$. See Fig.~\ref{criterion_fig}, in which some possible locations of $x$ are indicated by dots, and the corresponding collections $\{x,x_{i_1},\dotsc,x_{i_p}\}$ are also indicated next to the dots. Therefore we have a convex combination $\bar{x}=\lambda x+\lambda_1x_{i_1}+\dotsb+\lambda_px_{i_p}$ with $\lambda,\lambda_i\in (0,1)$. For the convex roof extension $\hat{f}$, we then have $f(\bar{x})\le \lambda \hat{f}(x)+\lambda_1f(x_{i_1})+\dotsb+\lambda_pf(x_{i_p})$. This gives the lower bound $\lambda^{-1}(f(\bar{x})-\lambda_1f(x_{i_1})-\dotsb-\lambda_pf(x_{i_p}))$ for the value $\hat{f}(x)$. Therefore $\hat{f}$ does not take the infinity value.
\end{proof}

\begin{figure}[h]
\centering
\begin{tikzpicture}[scale=1,>=latex]

	\draw
		(-1.5,-1.5)  -- (4.5,4.5)
		(-2,0) -- (6,0)
		(4.5,-1.5) -- (2.5,4.5)
		(1.5,-1.5) -- (3.5,4.5)
		(-2.1,-0.9) -- (5.6,2.4)
		(5.5,-0.9) -- (-1,3);
	\draw[dashed]
		(0,0) -- (1,3.8) -- (4,0) -- cycle;
	\draw[very thick]
		(0,0) -- (3,3) -- (4,0) -- cycle;
	
	\node at (0.15,-0.2) {\small $x_0$};
	\node at (3.85,-0.2) {\small $x_1$};
	\node at (3.3,2.95) {\small $x_2$};
	\node at (2.25,1.3) {\small $\bar{x}$};
	
	\fill (1,3.8) circle (0.1);
	\node at (0,3.9) {\small $\{x,x_0,x_1\}$};
	\node at (1.1,4.06) {\small $x$};
	
	\fill (0,2.4) circle (0.1);
	\node at (-0.3,2) {\small $\{x,x_1\}$};
	
	\fill (-1,0) circle (0.1);
	\node at (-1,0.4) {\small $\{x,x_1,x_2\}$};
	
	\fill (-1.4,-0.6) circle (0.1);
	\node at (-2.4,-0.4) {\small $\{x,x_1,x_2\}$};
	
	\fill (4,4.3) circle (0.1);
	\node at (4,4.7) {\small $\{x,x_0,x_1\}$};
	
	\fill (1.4,0.6) circle (0.1);
	\draw[<-] (1.36,0.5) -- (0.5,-0.8) node[below=-3] {\small $\{x,x_1,x_2\}$};
	
	\fill (2.16,0.5) circle (0.1);
	\draw[<-] (2.2,0.4) -- (3,-0.8) node[below=-3] {\small $\{x,x_2\}$};
	
	\fill (3,2) circle (0.1);
	\draw[<-] (3.1,2.05) -- (5,3) node[above=-3] {\small $\{x,x_0,x_1\}$};

\end{tikzpicture}
\caption{$\bar{x}$ is in the relative interior of the convex hull of $x,x_{i_1},\dotsc,x_{i_p}$.}
\label{criterion_fig}
\end{figure}
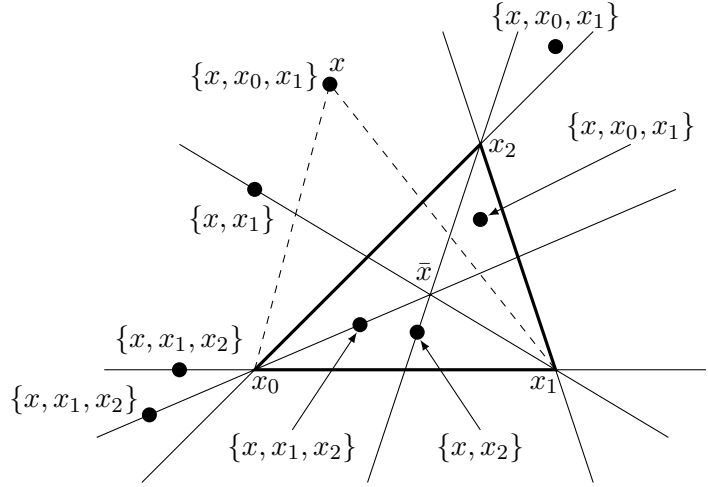

\begin{example}\label{eg22}
In terms of the polar coordinates on ${\bb R}^2$, consider $f=r^2+a\theta$ on the {\em connected} domain
\[
\Omega=\{(r,\theta)\colon |r-1|<\epsilon,\; 0<\theta<2\pi\}.
\]
Note that the Hessian $H_{r^2}(u,v)=2u^2+2v^2$ of $r^2=x^2+y^2$ has positive lower bound on $\Omega$. Moreover, for sufficiently small $a$ and $\epsilon$, the second order derivatives of $a\theta$ are uniformly small on $\Omega$. Therefore the Hessian $H_f=H_{r^2}+aH_{\theta}$ remains positive definite on $\Omega$, so that $f$ is locally convex on $\Omega$. In fact, $|a|<2(1-\epsilon)^2$ is enough.

If $f$ is convex, then by Theorem \ref{tohull}, $f$ extends to a convex function on the convex hull $\Omega^{\text{co}}=\{(x,y)\colon x^2+y^2<(1+\epsilon)^2\}$. In particular, by the continuity of convex functions (see \cite[Theorem 10.1]{Rockafellar}, for example), $f$ extends to a continuous function on $\Omega^{\text{co}}$. Since $\lim_{y\to 0^+}f(1,y)=1$ and $\lim_{y\to 0^-}f(1,y)=1+2\pi a$, we get a contradiction.

The example shows that even on a connected domain, the local convexity does not always imply the convexity.
\end{example}

\begin{example}\label{eg31}
The strict version of Theorem \ref{tohull} is not true. The function $f(x,y)=e^x y^2$ is strictly convex on $\Omega={\bb R}\times ({\bb R}-0)$. The continuity of convex functions implies that the only convex extension is $\hat{f}(x,y)=e^x y^2$ throughout ${\bb R}^2$. The extension is not strictly convex along the $x$-axis.
\end{example}

Despite the counterexample, there is still the possibility that the strict version of Theorem \ref{tohull} may hold for compact subsets. Indeed this is the case for the {\em Hessian convexity}.

We recall that, if the Hessian of a second order differentiable function is positive semidefinite at every point of the domain, then the function is locally convex. If the Hessian is positive definite, then the function is strictly locally convex. Here, in case the domain is not open, we mean that the second order differentiability happens on an open subset containing the domain. In particular, if the domain is a compact subset $\Omega$, then a $C^2$-function on $\Omega$ is a $C^2$-function on the $\epsilon$-neighborhood
\[
\Omega^{\epsilon}
=\{x\colon \|x-y\|<\epsilon\text{ for some }y\in\Omega\}
\]
for some $\epsilon>0$.

We note that the convex hull of the $\epsilon$-neighborhood $\Omega^{\epsilon,\text{co}}$ is the same as the $\epsilon$-neighborhood of the convex hull $\Omega^{\text{co},\epsilon}$.

\begin{theorem}\label{tohulldiff}
If a $C^2$-function $f$ is convex and has positive definite Hessian on an open neighborhood of a compact subset $\Omega\sub {\bb R}^n$, then $f|_{\Omega}$ can be extended to a $C^2$-function on the convex hull $\Omega^{\text{\rm co}}$ with positive definite Hessian.  
\end{theorem}

The proof shows that if the function is $C^r$, $r\ge 2$, then we can make the extended function $C^r$.

\begin{proof}
Suppose $f$ is convex and has positive definite Hessian on the $3\epsilon$-neighborhood $\Omega^{3\epsilon}$. By Theorem \ref{tohull}, we have a convex extension $\hat{f}$ on the convex hull $\Omega^{\text{co},3\epsilon}$. 

Let $\phi\ge 0$ be a smooth function supported on the ball $B_{\delta}$ of radius $\delta\in (0,\epsilon)$ and centered at the origin, such that $\int\phi(x)dx=1$. Then    
\[
g(x)=\int \hat{f}(y)\phi(x-y)dy=\int \hat{f}(x-y)\phi(y)dy 
\]
is a smooth function on $\Omega^{\text{co},2\epsilon}$. The second expression for $g$ and the convexity of $\hat{f}$ imply that $g$ is convex on $\Omega^{\text{co},2\epsilon}$. Let $0\le\alpha\le 1$ be a smooth function, such that $\alpha=1$ on $\Omega$ and $\alpha=0$ outside $\Omega^{\epsilon}$. Construct
\[
h=\alpha f+(1-\alpha)(g+c \|x\|^2)
=f+(1-\alpha)(g-f+c \|x\|^2),
\]
where $c>0$ is a very small constant to be determined. The function $h$ is $C^2$ on $\Omega^{\text{co},2\epsilon}$.

On $\Omega$, we have $h=f$. Therefore $h$ extends $f|_{\Omega}$.

On $\Omega^{\text{co},2\epsilon}-\Omega^{\epsilon}$, we have $h=g+c \|x\|^2$. Since $g$ is convex and $c>0$, the Hessian of $h$ is positive definite.

For $x\in \Omega^{\epsilon}$, by $\phi(y)=0$ for $\|y\|\ge\epsilon$ and $x-y\in \Omega^{2\epsilon}$ for $\|y\|<\epsilon$, we have
\begin{align*}
\pa_{ij}g(x)-\pa_{ij}f(x)
&=\int (\pa_{ij}\hat{f}(x-y)-\pa_{ij}f(x))\phi(y)dy \\
&=\int (\pa_{ij}f(x-y)-\pa_{ij}f(x))\phi(y)dy
\end{align*}
and the similar equalities for $\pa_ig-\pa_if$ and $g-f$. Since the derivatives of $f$ up to the second order are uniformly continuous on the compact subset $\overline{\Omega^{2\epsilon}}$, for sufficiently small $\delta$, the differences $|\pa_{ij}f(x-y)-\pa_{ij}f(x)|$, $|\pa_{i}f(x-y)-\pa_{i}f(x)|$, $|f(x-y)-f(x)|$ can be uniformly small for $x\in \Omega^{\epsilon}$ and $\|y\|<\delta$. Since $\phi$ is supported on $B_{\delta}$, we conclude that $\pa_{ij}g(x)-\pa_{ij}f(x)$, $\pa_ig-\pa_if$ and $g-f$ can also be uniformly small on $\Omega^{\epsilon}$. By further choosing $c$ to be sufficiently small (in addition to the already small $\delta$), the second order derivatives of $(1-\alpha)(g-f+c \|x\|^2)$ can be uniformly small on $\Omega^{\epsilon}$. 

On the other hand, the Hessian of $f$ is positive definite on the compact subset $\overline{\Omega^{\epsilon}}$ and therefore has a positive definite lower bound on $\Omega^{\epsilon}$. This means that $H_f(v)\ge C\|v\|^2$ on $\Omega^{\epsilon}$ for some constant $C>0$ and any vector $v$. By choosing sufficiently small $\delta$ and $c$, the absolute value of the Hessian of $\alpha(g-f+c \|x\|^2)$ is $<C\|v\|^2$ on $\Omega^{\epsilon}$. Then the Hessian of $h$ is still positive definite on $\Omega^{\epsilon}$.
\end{proof} 

\begin{example}
In Theorem \ref{tohulldiff}, the convexity is assumed on a neighborhood of $\Omega$. To sees why this cannot be weakened to the convexity on $\Omega$ only, consider $\Omega=\{0,1,2\}\sub{\bb R}$. Let the function $f$ on a neighborhood of $\Omega$ be given by $f(x)=(x-x_0)^2$ near any $x_0\in\Omega$. Then $f$ has positive definite Hessian, and is convex but not strictly convex on $\Omega$. Therefore any extension of $f$ to $\Omega^{\text{co}}=[0,2]$ cannot have positive definite Hessian at every point.
\end{example}

Despite the counterexample, it is conceivable to replace ``convex on a neighborhood of $\Omega$'' by ``strictly convex on $\Omega$'' in Theorem \ref{tohulldiff}. The key is the following: Suppose $f$ is a function on a neighborhood of compact $\Omega$, satisfying the following:
\begin{enumerate}
\item $f$ is strictly convex on $\Omega$.
\item For any $x\in \Omega$, $f$ is strictly convex on a neighborhood of $x$.
\end{enumerate}
Then I suspect that $f$ is strictly convex on a neighborhood of $\Omega$.

\section{Extension out of the Convex Hull}
\label{outside}

Not every convex function on a convex subset $\Omega$ can be extended to a convex function on the whole space without taking the infinity value. For example, a convex continuous function $f(x)$ on $[a,b]$ can be extended to a convex function on ${\bb R}$ without taking the infinity if and only if the one-sided derivatives $f'_+(a)$ and $f'_-(b)$ are finite. Equivalently, this means that $f$ is Lipschitz on the whole interval. The observation can be extended to multivariable convex functions. The following construction is essentially the same as the one by Dragomirescu and Ivan \cite{di} and is reformulated to avoid the infinity value. 

\begin{theorem}\label{toall}
A convex function on a bounded convex subset can be extended to a convex function on the whole linear space if and only if it is a Lipschitz function.
\end{theorem}

\begin{proof}
The necessity follows from \cite[Theorem 10.4]{Rockafellar}. For the sufficiency, we first consider the case that $\Omega$ affinely spans the whole space. Then any point $x\in {\bb R}^n-\Omega$ is of the form
\[
x=\lambda y+(1-\lambda)z,\quad y,z\in\Omega,\; \lambda>1.
\]
If $f$ extends to a convex function $\tilde{f}$ on the whole ${\bb R}^n$, then we must have
\[
\tilde{f}(x)\ge \lambda f(y)+(1-\lambda)f(z).
\]
Therefore we define
\[
\tilde{f}(x)
=\sup\{\lambda f(y)+(1-\lambda)f(z)\colon
x=\lambda y+(1-\lambda)z, \;y,z\in\Omega,\; \lambda\ge 1\}.
\]
The convexity of $f$ implies that $\tilde{f}$ extends $f$.

Suppose $|f(x)-f(x')|\le l\|x-x'\|$ on $\Omega$. Then for $y,z\in\Omega$, we have
\begin{align*}
|\lambda f(y)+(1-\lambda)f(z)|
&\le \lambda|f(y)-f(z)|+|f(z)| \\
&\le \lambda l\|y-z\|+|f(z)| \\
&= l\|x-z\|+|f(z)|.
\end{align*}
Since $\Omega$ is bounded and a Lipschitz function on bounded $\Omega$ is bounded, the right side is bounded for fixed $x$. This proves that $\tilde{f}$ does not take the infinity value. 

To prove $\tilde{f}$ is convex, we only need to prove that it is interval convex. Consider a convex combination $x=\mu_1x_1+\mu_2x_2$. Suppose $x=\lambda y+(1-\lambda)z$ for some $\lambda> 1$ and $y,z\in\Omega^{\text{ri}}$. Since $\Omega$ is assumed to affinely span the whole space, we have a small interval $I\sub\Omega$ such that $I$ is parallel to $[x_1,x_2]$ and $y$ is in the interior of $I$. Then we choose a point $z'$ in the interior of the interval $[y,z]$ sufficiently close to $y$, such that the intervals $[x_1,z']$ and $[x_2,z']$ intersect $I$ at $y_1$ and $y_2$. See Fig.~\ref{toall_fig}.

\begin{figure}[h]
\centering
\begin{tikzpicture}[scale=1]

	\draw
		(0,0) node[below] {\small $x_1$} -- 
		(2.5,1.5) node[above left =-4] {\small $z'$} -- 
		(5,0) node[below] {\small $x_2$} -- (0,0)
		(2,0) node[below] {\small $x$} -- 
		(3,3) node[left] {\small $z$}
		(1.5,1) node[below right=-2] {\small $y_1$} -- node[below left=-1] {\small $y$}
		(3.7,1) node[below left=-1] {\small $y_2$} node[right] {\small $I$};

\end{tikzpicture}
\caption{Prove the interval convexity of the extension.}
\label{toall_fig}
\end{figure}
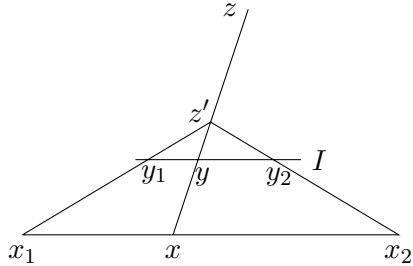

Write $x=\lambda' y+(1-\lambda')z'$. Then $\lambda'>1$, and by the convexity of $f$ on the interval $[x,z]$, on which $x,y,z',z$ form a monotone sequence, we have
\[
\lambda' f(y)+(1-\lambda')f(z')\geq \lambda f(y)+(1-\lambda)f(z).
\]
Since $I$ is parallel to $[x_1,x_2]$, we have $y=\mu_1y_1+\mu_2y_2$, $x_1=\lambda' y_1+(1-\lambda')z'$, $x_2=\lambda' y_2'+(1-\lambda')z'$. Then the convexity of $f$ on $I$ tells us
\[
f(y)\le \mu_1f(y_1)+\mu_2f(y_2).
\]
Moreover, by the definition of $\tilde{f}$, we have
\[
\tilde{f}(x_1)\ge \lambda' f(y_1)+(1-\lambda')f(z'),\quad 
\tilde{f}(x_2)\ge \lambda' f(y_2)+(1-\lambda')f(z').
\]
Consequently, 
\begin{align*}
\mu_1\tilde{f}(x_1)+\mu_2\tilde{f}(x_2)  
& \ge \lambda' (\mu_1f(y_1)+\mu_2f(y_2))+(1-\lambda')f(z')  \\
& \ge \lambda'f(y)+(1-\lambda')f(z') \\
& \ge \lambda f(y)+(1-\lambda)f(z).  
\end{align*}
The Lipschitz property of $f$ on $\Omega$ implies that it is continuous, so that the supremum of the right side for all $y,z\in\Omega^{\text{ri}}$ is the same as the supremum $\tilde{f}(x)$ for all $y,z\in\Omega$.

Finally, if $\Omega$ does not affinely span the whole space, then the argument above produces an extension to a convex function on the affine span of $\Omega$. It is very easy to further extend the convex function on an affine subspace to a convex function on the whole space.
\end{proof}

We may try to combine Theorems \ref{tohull} and \ref{toall}. The key is to verify that the extension to the convex hull is still Lipschitz. While this may not be true in general, the following shows that, if the domain is a ``convex boundary band'', then this is true.

\begin{theorem}\label{toalldiff}
Suppose $\Omega$ is a bounded convex subset, and $A$ is a subset satisfying $\bar{A}\sub\Omega^{\text{\rm ri}}$. Then a convex function on $\Omega-A$ can be extended to a convex function on the whole linear space if and only if it is Lipschitz.
\end{theorem}

\begin{proof}
The necessity follows from \cite[Theorem 10.4]{Rockafellar}. It remains to prove the sufficiency. 

By Theorem \ref{tohull}, a convex function $f$ on $\Omega-A$ can be extended to a convex function $\hat{f}$ on $(\Omega-A)^{\text{co}}=\Omega$ (the equality is a consequence of $\bar{A}\sub\Omega^{\text{\rm ri}}$). We will show that $|f(x)-f(y)|\le l\|x-y\|$ on $\Omega-A$ implies $|\hat{f}(x)-\hat{f}(y)|\le l\|x-y\|$ on $\Omega$. By Theorem \ref{toall}, $\hat{f}$ can be further extended to a convex function on the whole linear space.

Let $x,y\in \Omega^{\text{ri}}$. Let $L$ be the straight line connecting $x$ and $y$. Then $L\cap\bar{A}$ is a compact subset inside the open interval $L\cap  \Omega^{\text{ri}}$. Therefore we can find $x_1,y_1,x_2,y_2\in L\cap (\Omega^{\text{ri}}-\bar{A})$, such that $x_1,y_1,x,y,x_2,y_2$ form a strictly monotone sequence on $L$. See Fig.~\ref{toalldiff_fig}. Then by the convexity of $\hat{f}$ on the interval $[x_1,y_2]$, we have 
\[
\frac{|\hat{f}(x)-\hat{f}(y)|}{\|x-y\|}\leq
\max\left\{
\frac{|f(x_1)-f(y_1)|}{\|x_1-y_1\|},
\frac{|f(x_2)-f(y_2)|}{\|x_2-y_2\|}\right\}
\le l.
\]

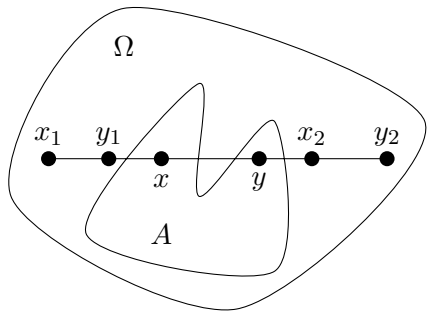
\begin{figure}[h]
\centering
\begin{tikzpicture}[scale=1]

\draw 
	plot[smooth cycle] coordinates{(-2,0.5) (-0.5,3) (3.5,1.5) (1,-1) }
	plot[smooth cycle] coordinates{(-1,0) (0.5,2) (0.5,0.5) (1.5,1.5) (1.5,-0.5) }
	(-1.5,1) -- (3,1);
\node at (-0.5,2.5) {\small $\Omega$};
\node at (0,0) {\small $A$};

\fill (-1.5,1) circle (0.1);
\node at (-1.5,1.3) {\small $x_1$};

\fill (-0.7,1) circle (0.1);
\node at (-0.7,1.3) {\small $y_1$};

\fill (0,1) circle (0.1);
\node at (0,0.7) {\small $x$};

\fill (1.3,1) circle (0.1);
\node at (1.3,0.7) {\small $y$};

\fill (2,1) circle (0.1);
\node at (2,1.3) {\small $x_2$};

\fill (3,1) circle (0.1);
\node at (3,1.3) {\small $y_2$};

\end{tikzpicture}
\caption{Verify the Lipschitz property.}
\label{toalldiff_fig}
\end{figure}

So we proved $|\hat{f}(x)-\hat{f}(y)|\le l\|x-y\|$ for any interior points $x$ and $y$ of the interval $L\cap\Omega$. If $x\not\in\Omega^{\text{ri}}$, then $x$ is an end point of the interval $L\cap\Omega$, and $x\in L\cap (\Omega-\bar{A})$. Therefore we can find a sequence $x_i\in L\cap (\Omega^{\text{ri}}-\bar{A})$ converging to $x$. Since $f$ is Lipschitz on $\Omega-\bar{A}$ and is therefore continuous at $x$, taking the limit of $|\hat{f}(x_i)-\hat{f}(y)|=|f(x_i)-\hat{f}(y)|\le l\|x_i-y\|$ gives us $|\hat{f}(x)-\hat{f}(y)|=|f(x)-\hat{f}(y)|\le l\|x-y\|$. Same argument can be made when $y\not\in\Omega^{\text{ri}}$, and even when both $x,y\not\in\Omega^{\text{ri}}$. 
\end{proof}

\begin{example}
To see the necessity of the condition $\bar{A}\sub\Omega^{\text{ri}}$ in Theorem \ref{toalldiff}, let $\Omega$ be the disk of radius $2$ centered at the origin $(0,0)$, and let $A$ be the disk of radius $1$ centered at the point $(1,0)$. Then $f(x,y)=\frac{1}{2-x}$ is convex on $\Omega-A$, but any convex extension that includes the point $(2,0)$ must take the value $+\infty$ at $(2,0)$.
\end{example}

Example \ref{eg31} can be slightly modified to show that Theorems \ref{toall} and \ref{toalldiff} do not hold for strictly convex functions. On the other hand, Theorem \ref{tohulldiff} suggests that it is still possible to consider the extension of $C^2$-functions with positive definite Hessian beyond the convex hull. For such Hessian convex functions, it is natural to allow non-convex domains.

\begin{theorem}\label{hessian}
Suppose $f$ is a $C^2$-function with positive definite Hessian on a compact subset $\Omega$. Then for any open subset $\Omega'$ containing $\Omega^{\text{\rm co}}$ and sufficiently small $\epsilon>0$, $f$ can be extended to a $C^2$-function on $\Omega^{\epsilon}\cup ({\bb R}^n-\Omega')$ with positive definite Hessian.  
\end{theorem}

The theorem suggests that $\Omega\cup({\bb R}^n-\Omega^{\text{co}})$ is roughly a ``universal extendable region'' for the Hessian convexity. However, the theorem can be used repeatedly to extend to bigger regions. The key point is that the Hessian convexity is a local property, so that we do not need to maintain the (global) convexity for the extension. For example, in Fig.~\ref{repeat}, we may start with a function with positive definite Hessian on $\Omega_1$, extend to a function with positive definite Hessian on $\Omega_2$, and then further extend to a function with positive definite Hessian on $\Omega_3$.

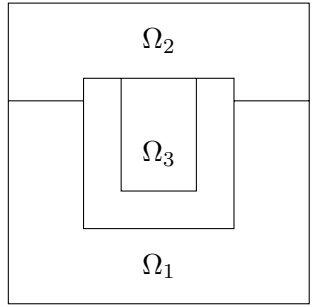
\begin{figure}[h]
\centering
\begin{tikzpicture}[scale=1]

\draw
	(-2,0) rectangle (2,4)
	(-1,1) rectangle (1,3)
	(-0.5,3) -- (-0.5,1.5) -- (0.5,1.5) -- (0.5,3)
	(-2,2.7) -- (-1,2.7)
	(2,2.7) -- (1,2.7);

\node at (0,0.5) {\small $\Omega_1$};
\node at (0,3.5) {\small $\Omega_2$};
\node at (0,2) {\small $\Omega_3$};

\end{tikzpicture}
\caption{Repeatedly extend the Hessian convexity.}
\label{repeat}
\end{figure}

\begin{proof}
Suppose $f$ has positive definite Hessian on the $2\epsilon$-neighborhood $\Omega^{2\epsilon}$. Since $\Omega$ is compact, we know $\Omega^{\text{co}}$ is compact, and can find finitely many big balls $B_i=B(x_i,r_i)$, such that 
\[
\Omega\sub \Omega^{\text{co}}\sub \cap B_i\sub \cap \bar{B}_i\sub \Omega'\cap \Omega^{\text{co},\epsilon}.
\]

\begin{figure}[h]
\centering
\begin{tikzpicture}[scale=1]

\draw
	plot[smooth cycle] coordinates{(-2,-1) (-2,1) (-1,1) (-1,0) (1,0) (1,1) (2,1) (2,-1)}
	(0,1.3) arc (90:80:20) node[below] {\small $B_3$}
	(0,1.3) arc (90:100:20) 
	(0,-1.28) arc (270:280:20) node[above] {\small $B_4$}
	(0,-1.28) arc (270:260:20) 
	(-2.4,-2.1) node[right] {\small $B_1$} arc (182:170:20)
	(2.4,-2.1) node[left] {\small $B_2$} arc (-2:10:20);

\draw[dashed]
	plot[smooth cycle] coordinates{(-2.3,-1.3) (-2.3,1.3) (-0.7,1.3) (-0.7,0.3) (0.7,0.3) (0.7,1.3) (2.3,1.3) (2.3,-1.3)}
	plot[smooth cycle] coordinates{(-2.6,-1.6) (-2.6,1.6) (-0.4,1.6) (-0.4,0.6) (0.4,0.6) (0.4,1.6) (2.6,1.6) (2.6,-1.6)}
	(-1.8,1.5) -- node[below=3] {\small $\Omega^{\text{co},\epsilon}$} (1.8,1.5);
	
\draw[<->]
	(0.5,-1.2) -- node[left] {\small $\epsilon$} (0.5,-1.57);
\draw[<->]
	(1,-1.18) -- node[right, fill=white] {\small $2\epsilon$} (1,-1.9);

\node at (0,-0.7) {\small $\Omega$};

\end{tikzpicture}
\caption{Extend the Hessian convexity outside the convex hull.}
\end{figure}
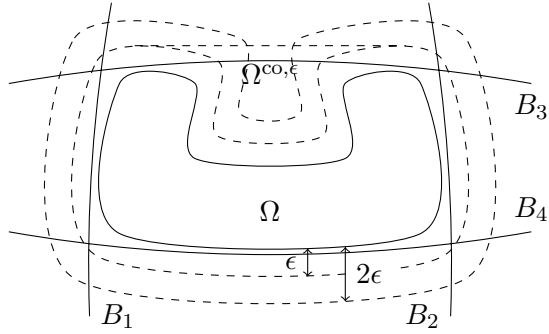

Let $\gamma(t)$ be a smooth function on $[0,\infty)$, such that $\gamma=0$ on $[0,1]$ and $t^{-1}\gamma\,'(t)$ is strictly increasing on $(1,\infty)$. The Hessian of $\gamma(\|x\|)$ is
\[
H_{\gamma(\|x\|)}(v)
=\frac{1}{\|x\|}\left(\gamma\,'(\|x\|)\|v\|^2+\left.\dfrac{d(t^{-1}\gamma\,'(t))}{dt}\right|_{t=\|x\|}(x\cdot v)^2\right),
\]
which is zero for $\|x\|\le 1$ and positive definite for $\|x\|>1$. Then 
\[
g(x)=\sum\gamma\left(\frac{\|x-x_i\|}{r_i}\right)
\] 
is a smooth function on ${\bb R}^n$, such that $g=0$ on $\cap B_i$. Moreover, the Hessian of $g$ is positive semidefinite on ${\bb R}^n$ and positive definite on ${\bb R}^n-\cap\bar{B}_i$.

Let $0\leq\alpha\leq 1$ be a smooth function on ${\bb R}^n$, such that $\alpha=1$ on $\Omega^{\epsilon}$ and $\alpha=0$ outside $\Omega^{2\epsilon}$. Then  
\[
h=\alpha f+cg
\]
is a $C^2$-function on ${\bb R}^n$. Here $c>0$ is a large constant to be determined.

Since $g=0$ on $\Omega\sub \cap B_i$, $h$ extends $f$. 

We have $h=f+cg$ on $\Omega^{\epsilon}$. Since the Hessian of $f$ is positive definite and the Hessian of $g$ is positive semidefinite, the Hessian of $h$ is positive definite.

The Hessian of $\alpha f$ is bounded on the compact subset $\overline{\Omega^{\text{co},2\epsilon}-\Omega'}$. The Hessian of $g$ is positive definite on $\overline{\Omega^{\text{co},2\epsilon}-\Omega'}\sub {\bb R}^n-\Omega'\sub {\bb R}^n-\cap\bar{B}_i$ and therefore has positive lower bound on the compact subset. By choosing sufficiently big $c$, the Hessian of $h$ is positive definite on $\Omega^{\text{co},2\epsilon}-\Omega'$.

We have $h=cg$ on ${\bb R}^n-\Omega^{\text{co},2\epsilon}$. Since the Hessian of $g$ is positive definite on ${\bb R}^n-\Omega^{\text{co},2\epsilon}\sub {\bb R}^n-\cap \bar{B}_i$, $h$ is positive definite on ${\bb R}^n-\Omega^{\text{co},2\epsilon}$. 
\end{proof}

A special case of Theorem \ref{hessian} is the convex boundary band in Theorem \ref{toalldiff}.

\begin{theorem}\label{hessian2}
Suppose $\Omega$ is a compact convex subset, and $A$ is a subset satisfying $\bar{A}\sub\Omega^{\text{\rm ri}}$. Then a $C^2$-function with positive definite Hessian on $\Omega-A$ can be extended to a $C^2$-function with positive definite Hessian on ${\bb R}^n-A$.  
\end{theorem}

In case $A$ is empty, the theorem says that a $C^2$-function with positive definite Hessian on a compact convex subset can be extended to a $C^2$-function with positive definite Hessian on the whole ${\bb R}^n$.

\section{Equivalence between Different Convexities}
\label{compare}

We know the three kinds of convexities are equivalent on convex subsets. The following shows that the equivalence still holds if the subset is not too far from being convex.  

\begin{theorem}\label{th201}
Suppose $\Omega$ is a convex subset, and $A$ is a closed subset of dimension $\le \dim \Omega-2$. Then for continuous functions on $\Omega-A$, the convexity, the local convexity and the interval convexity are equivalent.
\end{theorem}

The dimension can be defined in various ways, depending whether $A$ is a union of finitely many submanifolds, or a polyhedron, or some other topologically nice subset. The only key point we will use about the dimension condition is the consequence that topologists call ``general position''. This means that if $x,y$ are outside $A$, then we can move $y$ a little bit, so that the straight line connecting $x$ and $y$ avoids $A$. Such type of condition appeared in \cite{tab}, in which the main result is closely related to Theorems \ref{th201} and \ref{th202}.

The theorem basically concludes that the interval convexity on $\Omega-A$ implies the convexity. Since $A$ is closed and $\Omega^{\text{ri}}$ is relatively open, the interval convexity on $\Omega^{\text{ri}}-A$ implies the local convexity. By \cite[Theorem 10.1]{Rockafellar}, the function is always continuous on $\Omega^{\text{ri}}-A$. Therefore the continuity condition is really imposed on the boundary of $\Omega$.

\begin{proof}
We need to show that the property \eqref{convexi} implies the property \eqref{convexf}. The proof follows the classical proof that, on a convex subset, the property \eqref{convexf2} implies the property \eqref{convexf} (i.e., the case $k=2$ implies the general case).  

\medskip

\noindent{\em Step 1}: Review of the classical proof for the convex domain.

\medskip

The key point of the classical proof is to express a general convex combination as a sequence of convex combinations of two vectors. Consider a convex combination \eqref{ccombo} lying in $\Omega$. Without loss of generality, we may assume all $\lambda_i\in (0,1)$. Then the convex combination can be decomposed into convex combinations of two vectors
\begin{align*}
x &=\mu_1x_1+(1-\mu_1)y_1,\\
y_1 &=\mu_2x_2+(1-\mu_2)y_2,\\
&\;\vdots \\
y_{k-2} &=\mu_{k-1}x_{k-1}+(1-\mu_{k-1})y_{k-1}, \\
y_{k-1} &= x_k,
\end{align*}
where $\lambda_i=(1-\mu_1)\dotsb(1-\mu_{i-1})\mu_i$, $\mu_i\in (0,1)$ for $i<k$, and $\mu_k=1$. In the subsequent argument, the convex combination is fixed, so that $\lambda_i$ and $\mu_i$ are all fixed.

Applying \eqref{convexf2} to each convex combination above, we have
\begin{align*}
f(x) &\le \mu_1f(x_1)+(1-\mu_1)f(y_1),\\
f(y_1) &\le \mu_2f(x_2)+(1-\mu_2)f(y_2),\\
&\; \vdots \\
f(y_{k-2}) &\le \mu_{k-1}f(x_{k-1})+(1-\mu_{k-1})f(y_{k-1}), \\
f(y_{k-1}) &= f(x_k).
\end{align*}
Combining all the inequalities together, we get 
\begin{align*}
f(x) &\le \mu_1f(x_1)+(1-\mu_1)f(y_1) \\
&= \lambda_1f(x_1)+(1-\mu_1)f(y_1) \\
&\le \lambda_1f(x_1)+(1-\mu_1)[\mu_2f(x_2)+(1-\mu_2)f(y_2)] \\
&= \lambda_1f(x_1)+\lambda_2f(x_2)+(1-\mu_1)(1-\mu_2)f(y_2) \\
&\; \vdots \\
&\le \lambda_1f(x_1)+\lambda_2f(x_2)+\cdots+(1-\mu_1)\cdots(1-\mu_{k-1})f(y_{k-1})  \\
&=\lambda_1f(x_1)+\lambda_2f(x_2)+\cdots+\lambda_kf(x_k).
\end{align*}

\medskip

\noindent{\em Step 2}: Approximate convex combination decomposition.

\medskip

Under the assumption of the theorem, however, we only know $f$ is interval convex on $\Omega-A$. This means that \eqref{convexf2} holds only when the interval $[x_i,y_i]$ is contained in $\Omega-A$. Although this may not always hold, we will argue that for any $\epsilon>0$, there are $x_i',y_i'$ satisfying
\[ 
\|x_i-x_i'\|<\epsilon,\quad
[x_i',y_i']\in \Omega-A,
\]
such that 
\begin{align*}
x &=\mu_1x_1'+(1-\mu_1)y_1',\\
y_1' &=\mu_2x_2'+(1-\mu_2)y_2',\\
&\;\vdots \\
y_{k-2}' &=\mu_{k-1}x_{k-1}'+(1-\mu_{k-1})y_{k-1}', \\
y_{k-1}' &= x_k'.
\end{align*}
Then the classical proof gives us
\[
f(x)\le \lambda_1f(x_1')+\lambda_2f(x_2')+\cdots+\lambda_kf(x_k').
\]
As $\epsilon\to 0$, by the continuity of $f$ at $x_i$, the limit of the inequality gives \eqref{convexf}.

\begin{figure}[h]
\centering
\begin{tikzpicture}[scale=1]

	\draw
		(0,0) node[below=-2] {\small $x_1$} -- 
		(4,0) node[below=-2] {\small $x_2$} -- 
		(3,3) node[above] {\small $x_3$} -- cycle
		(0,0) -- (3.6,1.2) node[above left=-2] {\small $y_1$};
	\draw[dashed]
		(-0.2,0.3) node[left=-2] {\small $x_1'$} -- (3.9,0.95) node[right] {\small $y_1'$}
		(4.2,-0.1) node[right=-2] {\small $x_2'$} -- 
		(3.4,2.7) node[above right=-4] {\small $y_2'=x_3'$};
	\node at (1.8,0.8) {\small $x$};

\end{tikzpicture}
\caption{Approximate convex combination, with dashed lines avoiding $A$.}
\end{figure}
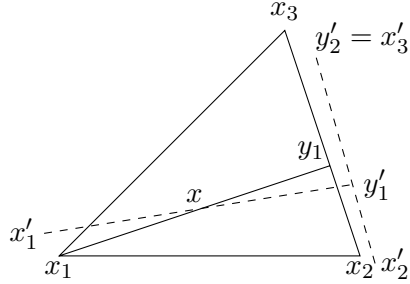

\medskip

\noindent{\em Step 3}: Construction of the approximation in the relative interior.

\medskip

We first construct the approximation for the case $x,x_1,\dotsc,x_k\in \Omega^{\text{ri}}-A$. The cases that some points may lie in $\Omega-\Omega^{\text{ri}}$ will be the limit of this first case.

The first approximation $x=\mu_1x_1'+(1-\mu_1)y_1'$ asks us to move $x_1$ to $x_1'\in \Omega$ by a small distance, such that the line connecting $x$ to $x_1'$ avoids $A$, and $y_1'=\frac{x-\mu_1x_1'}{1-\mu_1}$ still lies in $\Omega$. Since the dimension of $A$ is $\le \dim\Omega-2$, by the general position in topology, it is possible to move $x_1$ to $x_1'\in \Omega$ by arbitrarily small distance, such that the line avoids $A$. The small move of $x_1$ implies the correspondingly small move of $y_1$. If we start with $x_1,y_1\in \Omega^{\text{ri}}$, then we may choose the move of $x_1$ to be so small that both $x_1',y_1'$ still lie in $\Omega^{\text{ri}}$.

Inductively, after moving $x_{i-1}$ to $x_{i-1}'$ by a small distance, we get $y_{i-1}'$. Then we may move $x_i'$ by a small distance, such that the line connecting $y_{i-1}'$ and $x_i'$ avoids $A$. Moreover, this produces $y_i'=\frac{y_{i-1}'-\mu_ix_i'}{1-\mu_i}$ satisfying
\begin{align*}
\|y_i-y_i'\|
&\le \dfrac{1}{1-\mu_i}(\mu_i\|x_i-x_i'\|+\|y_{i-1}-y_{i-1}'\|) \\
&\le \dfrac{\mu_i}{1-\mu_i}\|x_i-x_i'\|+\dfrac{\mu_{i-1}}{(1-\mu_i)(1-\mu_{i-1})}\|x_{i-1}-x_{i-1}'\| \\
&\quad +\dotsb+\dfrac{\mu_1}{(1-\mu_i)\dotsb(1-\mu_1)}\|x_1-x_1'\|.
\end{align*}
By $y_i\in\Omega^{\text{ri}}$, there is $\epsilon'>0$, such that $y\in \Omega^{\text{aff}}$ and $\|y_i-y\|<\epsilon'$ for some $i$ implies $y\in\Omega^{\text{ri}}$. Then the estimation above for $\|y_i-y_i'\|$ shows that there is $\epsilon''>0$, such that 
\[
\|x_i-x_i'\|<\epsilon'' \text{ for all }i\le j
\implies
\|y_j-y_j'\|<\epsilon'.
\]
So by making sure that the following is always satisfied in the inductive construction,
\[
\|x_i-x_i'\|<\epsilon,\quad
\|x_i-x_i'\|<\epsilon'',\quad
x_i'\in \Omega^{\text{ri}},
\]
we get $\|x_i-x_i'\|<\epsilon$ and $x_i',y_i'\in \Omega^{\text{ri}}$ at the end. By the convexity of $\Omega$, we have $[x_i',y_i']\sub\Omega^{\text{ri}}$. Moreover, as part of the line connecting $y_{i-1}'$ and $x_i'$, the interval $[x_i',y_i']$ avoids $A$. Thus we conclude that $[x_i',y_i']\sub\Omega^{\text{ri}}-A$.

So far we proved the convexity of $f$ on $\Omega^{\text{ri}}-A$. By Theorem \ref{tohull}, we get a convex extension $\hat{f}$ on the convex hull of $\Omega^{\text{ri}}-A$. Due to the low dimension of $A$, the convex hull is $\Omega^{\text{ri}}$.

\medskip

\noindent{\em Step 4}: The case $x\in \Omega^{\text{ri}}$ but $x_i$ may not be in $\Omega^{\text{ri}}$.

\medskip

We may scale the convex combination $x=\lambda_1 x_1+\dotsb+\lambda_k x_k$ by a factor of $1-\delta$ and get
\[
x=\lambda_1 x_1'+\dotsb+\lambda_k x_k',\quad
x_i'=(1-\delta)x_i+\delta x.
\]
The idea is similar to the proof of Proposition \ref{sconvex}, except the convex combination was shrunken to be within a neighborhood of $x$ in the earlier argument, while here the convex combination is shrunken just a little bit. Since $x\in \Omega^{\text{ri}}$, we have $x_i'\in\Omega^{\text{ri}}$. Then the convex extension $\hat{f}$ on $\Omega^{\text{ri}}$ satisfies
\[
f(x)=\hat{f}(x)\le\lambda_1 \hat{f}(x_1')+\dotsb+\lambda_k \hat{f}(x_k').
\]
Next we will argue that $\lim_{\delta\to 0^+}\hat{f}(x_i')\le f(x_i)$ for each $x_i$. Then taking the limit of the inequality above gives us the inequality \eqref{convexf}.

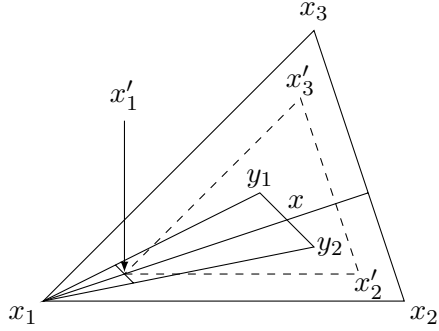
\begin{figure}[h]
\centering
\begin{tikzpicture}[scale=1.2,>=latex]

	\draw
		(0,0) node[below left=-2] {\small $x_1$} -- 
		(4,0) node[below right=-2] {\small $x_2$} -- 
		(3,3) node[above] {\small $x_3$} -- cycle
		(0,0) -- (3.6,1.2);
	\draw[dashed, shift={(0.9cm,0.3cm)},scale=0.65]
		(0,0) -- 
		(4,0) node[below right=-6] {\small $x_2'$} -- 
		(3,3) node[above=-3] {\small $x_3'$} -- cycle;
	\draw
		(0,0) -- 
		(2.4,1.2) node[above=-2] {\small $y_1$} -- 
		(3,0.6) node[right=-3] {\small $y_2$} --  cycle
		(1,0.2) -- (0.9,0.3) -- (0.8,0.4);
	\node at (2.8,1.1) {\small $x$};
	\draw[->]
		(0.9,2) node[above] {\small $x_1'$} -- (0.9,0.32);

\end{tikzpicture}
\caption{Prove $\lim_{\delta\to 0^+}f(x_1')\le f(x_1)$ when $[x_1,x]$ may intersect $A$.}
\end{figure}

If $[x_i,x]\sub \Omega-A$, then the restriction of $f$ on the interval $[x_i,x]$ is convex, and we get $\lim_{\delta\to 0^+}f(x_i')\le f(x_i)$ by the property of convex functions on closed interval. In case $[x_i,x]$ intersects $A$, however, another approximation argument is needed. In the subsequent argument, $i$ is fixed to be $1$.

By $x\in \Omega^{\text{ri}}-A$ and $A$ closed, there is an open ball $B\sub\Omega^{\text{ri}}-A$ around $x$. Then we find a convex combination 
\[
x=\mu_1y_1+\dotsb+\mu_dy_d,
\]
where $d=\dim\Omega$, $x\ne y_j\in B$, and $x_1,y_1,\dotsc,y_d$ are affinely independent. By the general position in topology, we can move each $y_j$ to $y_j'\in B$ by a small distance, so that the line $L_j$ connecting $x_1$ and $y_j'$ avoids $A$. By choosing $y_j'$ instead to be the intersection of $L_j$ with the affine subspace spanned by $y_1,\dots,y_d$, we may further make sure that $x$ is still a convex combination of $y_j'$. So without loss of generality, we may additionally assume that the convex combination $x=\mu_1y_1+\dotsb+\mu_dy_d$ satisfies $[x_1,y_j]\sub\Omega-A$ for all $j$. Then by the convexity of $f$ on the interval $[x_1,y_j]$, we get 
\[
\lim_{\delta\to 0^+}f((1-\delta)x_1+\delta y_j)\le f(x_1).
\]
On the other hand, by applying the convexity of $\hat{f}$ on $\Omega^{\text{ri}}$ to the convex combination in $\Omega^{\text{ri}}$
\begin{align*}
x_1'
&=(1-\delta)x_1+\delta(\mu_1y_1+\dotsb+\mu_dy_d) \\
&=\mu_1((1-\delta)x_1+\delta y_1)+\dotsb+\mu_d((1-\delta)x_1+\delta y_d), 
\end{align*}
we get 
\begin{align*}
\hat{f}(x_1')
&\le \mu_1\hat{f}((1-\delta)x_1+\delta y_1)+\dotsb+\mu_d\hat{f}((1-\delta)x_1+\delta y_d) \\
&= \mu_1f((1-\delta)x_1+\delta y_1)+\dotsb+\mu_df((1-\delta)x_1+\delta y_d). 
\end{align*}
Here is equality is due to $(1-\delta)x_1+\delta y_i\in \Omega^{\text{ri}}-A$. Taking the limit, we get
\[
\lim_{\delta\to 0^+}\hat{f}(x_1')
\le \mu_1f(x_1)+\dotsb+\mu_df(x_1)
= f(x_1).
\]

\medskip

\noindent{\em Step 5}: The case $x$ is not in $\Omega^{\text{ri}}$.

\medskip

Pick a point $z\in \Omega^{\text{ri}}$. For any $\delta>0$, the convex combination \eqref{ccombo} is approximated by
\[
x'=\lambda_1 x_1'+\dotsb+\lambda_k x_k',\quad
x_i'=(1-\delta)x_i+\delta z.
\]
By $z\in \Omega^{\text{ri}}$, we have $x_i',x'\in \Omega^{\text{ri}}$, so that 
\[
\hat{f}(x')\le\lambda_1 \hat{f}(x_1')+\dotsb+\lambda_k\hat{f}(x_k').
\]
It remains to show $\lim_{\delta\to 0^+}\hat{f}(x')=f(x)$ and the similar limits for $\hat{f}(x_i')$. Then we get \eqref{convexf} by taking the limit of the inequality above.

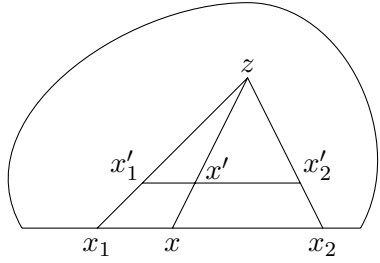
\begin{figure}[h]
\centering
\begin{tikzpicture}[scale=1]

	\draw
		(-1,0) to[out=120,in=180] (2,3) to[out=0,in=60] (3.5,0)
		(-1,0) -- (3.5,0)
		(0,0) node[below] {\small $x_1$} -- 
		(2,2) node[above=-2] {\small $z$} -- 
		(3,0) node[below] {\small $x_2$}
		(1,0) node[below] {\small $x$} -- (2,2)
		(0.6,0.6) node[above left=-3] {\small $x'_1$} -- 
		(2.7,0.6) node[above right=-3] {\small $x'_2$};
	\node at (1.6,0.8) {\small $x'$};

\end{tikzpicture}
\caption{Convexity when $x$ is not in $\Omega^{\text{ri}}$.}
\end{figure}

In fact, we will prove $\lim_{x'\in\Omega^{\text{ri}},\; x'\to x}\hat{f}(x')=f(x)$, which implies what we want. By the continuity of $f$ at $x\in \Omega-A$, for any $\epsilon>0$, there is $\delta>0$, such that $y\in \Omega-A$ and $\|y-x\|<\delta$ imply $|f(y)-f(x)|<\epsilon$. Now for $x'\in \Omega^{\text{ri}}$ satisfying $\|x'-x\|<\delta$, by the low dimension of $A$ and the continuity of $\hat{f}$ at $x'\in \Omega^{\text{ri}}$ (because convexity implies continuity in the relative interior), we can find $y$, such that
\[
y\in\Omega^{\text{ri}}-A,\quad 
\|y-x\|<\delta,\quad
|f(y)-\hat{f}(x')|=|\hat{f}(y)-\hat{f}(x')|<\epsilon.
\]
Then 
\[
|\hat{f}(x')-f(x)|\le |f(y)-\hat{f}(x')|+|f(y)-f(x)|<2\epsilon.
\]
This completes the proof.
\end{proof}

As remarked right before the proof, the interval convexity implies the continuity on $\Omega^{\text{ri}}-A$. Therefore the continuity of the function is used only in the fifth step of the proof. Since $x\not\in \Omega^{\text{ri}}$ necessarily forces $x_i\not\in \Omega^{\text{ri}}$, we may add a condition to avoid the fifth step and get the non-strict part of the following result.

\begin{theorem}\label{th202}
Suppose $\Omega$ is a convex subset with the additional property that, in every convex combination \eqref{ccombo} with $x_i,x\in\Omega$ and $x_i\ne x$, we always have $x\in \Omega^{\text{\rm ri}}$. Suppose $A$ is a closed subset of dimension $\le \dim \Omega-2$. Then the convexity, the local convexity and the interval convexity are equivalent on $\Omega-A$. Moreover, the strict versions of the convexity are also equivalent on $\Omega-A$. 
\end{theorem}

Open convex subsets have the additional property. Convex subsets with strictly convex boundaries also have the additional property. 

\begin{example}
Let $\Omega$ be the open square $(0,1)\times(0,1)$ together with a discrete subset $D\sub (0,1)\times 0$ of one side of the cube. Let $A=\emptyset$. Then the condition for $\Omega-A$ in Theorem \ref{th201} is satisfied. On the other hand, a function that takes $0$ on the open square and takes arbitrary positive values on $D$ is always locally convex but may not be convex.

The example shows that, given the lack of continuity assumption, the additional property in Theorem \ref{th202} is really needed. 
\end{example}

\begin{proof}[Proof of Theorem \ref{th202}]
We only need to explain the strict part. With the additional property, we only need to consider the case $x\in \Omega^{\text{ri}}-A$, and may additionally assume $\lambda_i\in (0,1)$. Given the interval convexity, we already know that the non-strict inequality \eqref{convexf} holds. To enhance the inequality to become strict, we move $x$ a little bit toward and away from $x_1$.

For a small $\delta>0$, introduce
\begin{align*}
x'
&=\delta x_1+(1-\delta)x
=x-\delta(x-x_1) \\
&=(\delta+(1-\delta)\lambda_1)x_1+(1-\delta)\lambda_2x_2+\cdots+(1-\delta)\lambda_kx_k, \\
x''
&=-\delta x_1+(1+\delta)x
=x+\delta(x-x_1) \\
&=(-\delta+(1+\delta)\lambda_1)x_1+(1+\delta)\lambda_2x_2+\cdots+(1+\delta)\lambda_kx_k.
\end{align*}
If $\delta$ is small enough, then the above are still convex combinations. Moreover, by $x\in \Omega^{\text{ri}}-A$, we can also make sure that $[x',x'']\sub\Omega^{\text{ri}}-A$.

By the convexity that we have already proved, we have
\begin{align*}
f(x')
&\le (\delta+(1-\delta)\lambda_1)f(x_1)+(1-\delta)\lambda_2f(x_2)+\cdots+(1-\delta)\lambda_kf(x_k), \\
f(x'')
&\le (-\delta+(1+\delta)\lambda_1)f(x_1)+(1+\delta)\lambda_2f(x_2)+\cdots+(1+\delta)\lambda_kf(x_k).
\end{align*}
Then by $x=\frac{1}{2}(x'+x'')$ and the strict convexity of $f$ on $[x',x'']$, we get $f(x)<\frac{1}{2}(f(x')+f(x''))$. Combining the three inequalities together, we get \eqref{convexf} with strict inequality.
\end{proof}

Can we relax the dimension condition in Theorems \ref{th201} and \ref{th202}? The following says that the condition cannot be relaxed to $\dim A\le \dim\Omega-1$.

\begin{theorem}\label{th203}
Suppose $U$ is an $(n-1)$-dimensional submanifold of ${\bb R}^n$, and $\Omega\sub {\bb R}^n-U$ is a subset. Suppose $L$ is a straight line that intersects $U$ transversely at an interior point $u$ of $U$. If there is a point $p\in L\cap \Omega$ and an open interval $I\sub L\cap\Omega$ lying in different components of $L-u$, then there is a locally convex function on $\Omega$ that is not convex.
\end{theorem}

The setup for the theorem is illustrated in Fig.~\ref{th203_fig}. The statement uses some standard topological concepts. For those who are not familiar with the concepts, an affinely equivalent description is given in the first paragraph of the proof. We also note that the proof actually constructs a function $f$ on the complement of any $(n-1)$-dimensional submanifold, such that $f$ has positive definite Hessian but is not convex.

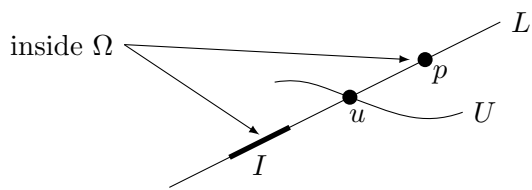
\begin{figure}[h]
\centering
\begin{tikzpicture}[scale=1,>=latex]

	\draw
		(-0.4,-0.2) --  
		(4,2) node[right] {\small $L$}
		(1,1.2) to[out=10,in=160] (2,1) to[out=-20,in=200] (3.5,0.8) node[right] {\small $U$};
	\draw[line width=2]
		(0.4,0.2) -- node[below] {\small $I$} (1.2,0.6);
	\draw[->]
		(-1,1.7) node[left] {\small inside $\Omega$} -- (2.8,1.5);
	\draw[->]
		(-1,1.7) -- (0.8,0.5);
		
	\fill (3,1.5) circle (0.1);
	\node at (3.2,1.3) {\small $p$};
	
	\fill (2,1) circle (0.1);
	\node at (2.1,0.75) {\small $u$};

\end{tikzpicture}
\caption{Condition for local convexity not implying convexity.}
\label{th203_fig}
\end{figure}

\begin{example}
The existence of the interval $I$ in Theorem \ref{th203} cannot be reduced to a point. For example, if $\Omega$ consists of exactly two points, then any function on $\Omega$ is convex. More generally, for convex subsets $\Omega_i\sub {\bb R}^{n_i}$ and $a_i\not\in \Omega_i$, any locally convex function on $\Omega=(\Omega_1\times a_2)\cup (a_1\times\Omega_2) \sub{\bb R}^{n_1+n_2}$ is convex. 
\end{example}

\begin{proof}[Proof of Theorem \ref{th203}]
Write points in ${\bb R}^n$ as $x=(t,y)$, $t\in {\bb R}$, $y\in {\bb R}^{n-1}$. By an affine transformation, we may assume that $U$ contains the graph $G=\{(g(y),y)\colon y\in B\}$ of a continuous function $g$ over the unit ball $B=\{y\colon \|y\|<1\}$. We may also assume that there is $a$ satisfying 
\[
-a<g(0)<a,\quad
p=(a,0)\in \Omega,\quad
I=(-a-\delta,-a+\delta)\times 0\sub \Omega. 
\]
See Fig.~\ref{th203_fig2}.

Let $\beta(t)$ be a smooth function, such that $\beta=1$ on $(-\infty,\frac{1}{2}]$ and $\beta=0$ on $[1,+\infty)$. Let
\[
f(t,y)=\begin{cases}
t^2+b\|y\|^2+(t+c)\beta(\|y\|^2),
&\text{ for }t>g(y),\; y\in B, \\
t^2+b\|y\|^2,
&\text{ otherwise in }{\bb R}^n-G,
\end{cases}
\]
where $b$ and $c$ are constants to be determined. The function is smooth on ${\bb R}^n-G$. We have
\[
\frac{f(a,0)-f(-a,0)}{a-(-a)}=\frac{a+c}{2a},\quad
\frac{\pa f}{\pa t}(-a,0)=-2a.
\]
If we fix $c$ satisfying $c<-a-4a^2$, then 
\[
\frac{f(a,0)-f(-a,0)}{a-(-a)}<\frac{\pa f}{\pa t}(-a,0).
\]
This implies that $f$ is not convex on $p\cup I\sub \Omega$. Note that the interval $I\sub\Omega$ is needed here because of the use of the partial derivative $\frac{\pa f}{\pa t}$ at $(-a,0)$.

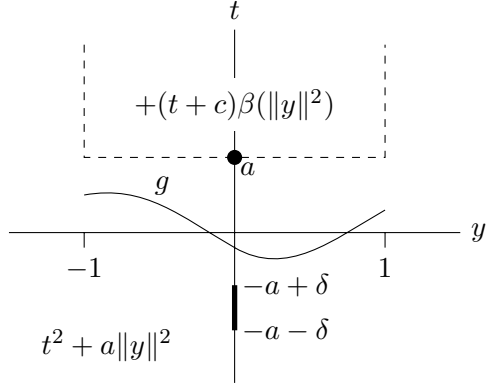
\begin{figure}[h]
\centering
\begin{tikzpicture}[scale=1]

	\draw
		(-3,0) -- (3,0) node[right] {\small $y$}
		(0,-2) -- (0,2.7) node[above] {\small $t$}
		(-2,0) -- +(0,-0.2) node[below] {\small $-1$}
		(2,0) -- +(0,-0.2) node[below] {\small $1$}
		%(0,-1) -- +(0.2,0) node[right] {\small $-a$}
		(-2,0.5) to[out=10,in=150] node[above] {\small $g$} (0,-0.2) to[out=-30,in=210] (2,0.3);
	\draw[dashed]
		(-2,2.5) -- (-2,1) -- (0,1) node[below right=-2] {\small $a$} -- (2,1) -- (2,2.5);
	\fill (0,1) circle (0.1);
	\draw[line width=2]
		(0,-1.3) node[right=-2] {\small $-a-\delta$} -- +(0,0.6) node[right=-2] {\small $-a+\delta$};
	\node at (-1.7,-1.5) {\small $t^2+a\|y\|^2$};
	\node[fill=white] at (0,1.7) {\small $+(t+c)\beta(\|y\|^2)$};

\end{tikzpicture}
\caption{Construct locally convex but not convex function.}
\label{th203_fig2}
\end{figure}

It remains to choose $b$, so that $f$ has positive definite Hessian. The function $t^2+b\|y\|^2$ has positive definite Hessian as long as $b>0$. For $t>g(y)$ and $y\in B$, the Hessian of $f$ at $(t,y)$ is
\begin{align*}
\frac{1}{2}H_f(\tau,\eta)
&=\tau^2+b\eta^2+2\tau\beta'(\|y\|^2)y\cdot \eta \\ 
&\quad +(t+c)\beta'(\|y\|^2)\eta^2+2(t+c)\beta''(\|y\|^2)(y\cdot\eta)^2 \\
&=[\tau+\beta'(\|y\|^2)y\cdot \eta]^2 +b\eta^2 \\ 
&\quad+(t+c)\beta'(\|y\|^2)\eta^2+[2(t+c)\beta''(\|y\|^2)-\beta'(\|y\|^2)^2](y\cdot\eta)^2.
\end{align*}
Inside a big ball $B_R$ of radius $R$ and centered at the origin, the sum $(t+c)\beta'(\|y\|^2)\eta^2+[2(t+c)\beta''(\|y\|^2)-\beta'(\|y\|^2)^2](y\cdot\eta)^2$ does not involve $\tau$ and is bounded (meaning that the absolute value $\le C\|\eta\|^2$ for some constant $C$). Then by choosing $b$ to be bigger than this bound, we get the Hessian of $f$ to be positive definite on $B_R-G$.

By Theorem \ref{hessian2}, the function $f|_{B_R-G}$ can be extended to a function with positive definite Hessian on ${\bb R}^n-G$. Since ${\bb R}^n-G\supset {\bb R}^n-U\supset \Omega$, the extension is also a function with positive definite Hessian on $\Omega$.
\end{proof}

\end{document}